\documentclass{amsart}
\usepackage[all]{xy}
\usepackage{latexsym}
\usepackage{amsmath}
\usepackage{amssymb,amscd}
\usepackage{setspace}
\usepackage{mathdots}
\usepackage[mathscr]{eucal}
\usepackage{bbm}
\usepackage{stmaryrd}
\usepackage{color}

\newtheorem{theorem}{Theorem}[section] 
\newtheorem{lemma}[theorem]{Lemma}     
\newtheorem{corollary}[theorem]{Corollary}
\newtheorem{proposition}[theorem]{Proposition}
\newtheorem{remark}{Remark}
\newtheorem{definition}{Definition}
\newtheorem{conjecture}{Conjecture}

\numberwithin{equation}{section}

\newcommand{\p}{\mathfrak{p}}
\newcommand{\m}{\mathfrak{m}}

\newcommand{\cO}{\mathcal{O}}

\newcommand{\cH}{\mathcal{H}}

\newcommand{\F}{\mathbb F}

\newcommand{\Q}{\mathbb Q}
\newcommand{\R}{\mathbb R}

\newcommand{\Z}{\mathbb Z}

\newcommand{\A}{\mathbb A}

\newcommand{\rInj}{\mathrm{Inj}}
\newcommand{\rsoc}{\mathrm{soc}}

\newcommand{\red}{\mathrm{red}}

\newcommand{\Sp}{\mathrm{Sp}}

\newcommand{\Ord}{\mathrm{Ord}}

\newcommand{\ra}{\rightarrow}

\newcommand{\End}{\mathrm{End}}
\newcommand{\Hom}{\mathrm{Hom}}
\newcommand{\Ext}{\mathrm{Ext}}
\newcommand{\GL}{\mathrm{GL}}
\newcommand{\SL}{\mathrm{SL}}

\newcommand{\Gal}{\mathrm{Gal}}
\newcommand{\id}{\mathrm{Id}}
\newcommand{\bFp}{\overline{\F}_p}
\newcommand{\bQp}{\overline{\Q}_p}

\newcommand{\Sym}{\mathrm{Sym}}

\newcommand{\Ind}{\mathrm {Ind}}

\providecommand{\cInd}{\mathrm{c}\textrm{-}\mathrm{Ind}}

\newcommand{\brho}{\overline{\rho}}

\newcommand{\summ}{\displaystyle\sum\limits}

\providecommand{\ligne}{\textbf{---}}

\newcommand{\matr}[4]{\begin{pmatrix}{#1}&{#2}\\ {#3}&{#4}\end{pmatrix}}
\newcommand{\smatr}[4]{\bigl(\begin{smallmatrix} {#1}& {#2}\\ {#3}&{#4}\end{smallmatrix}\bigl)}

\def\red#1{\textcolor{red}{#1}}

\begin{document}

\title[An application of a theorem of Emerton]
{An application of a theorem of Emerton to mod $p$ representations of $\GL_2$}

\author{Yongquan HU }
\date{}

\maketitle

\begin{abstract}
Let $p$ be a prime and $L$ be a finite extension of $\Q_p$. We study the ordinary parts of   $\GL_2(L)$-representations arised in the mod $p$ cohomology of  Shimura curves attached to indefinite division algebras which splits at a finite place above $p$. 
The main tool of the proof is a theorem of Emerton \cite{Em3}. 
\end{abstract}
\tableofcontents

\section{Introduction}

Let $p$ be a prime number and $L$ a finite extension of $\Q_p$. Let $G=\GL_2(L)$. If $L=\Q_p$, the work of Barthel-Livn\'e \cite{BL} and of Breuil \cite{Br03} gave a complete classification of  irreducible smooth representations of $G$ over $\bFp$ with a central character (this last restriction is now removed by Berger \cite{Be}). However, when $F\neq \Q_p$, the situation is much more complicated and a large part of the theory is still mysterious. The main difficulty lies in the study of supersingular representations, which are  irreducible smooth representations of $G$ which do not arise as subquotients of  parabolic inductions. For example, when $[L:\Q_p]=2$, Schraen has shown that supersingular representations are not of finite presentation  \cite{Sc} (similar result was first proved in \cite{Hu12} if $L$ is a finite extension of $\F_p[[t]]$).  

Though the general theory of smooth representations of $G$ could be very weird,  when $L$ is unramified over $\Q_p$,   Breuil and Pa\v{s}k\=unas were able to construct some `nicer' ones in \cite{BP}. Precisely, they constructed families of admissible smooth representations (by local methods) which are related to two dimensional continuous $\bFp$-representations of $\Gal(\bQp/L)$ via the Buzzard-Diamond-Jarvis conjecture \cite{BDJ}. Recent work of Emerton-Gee-Savitt \cite{EGS}
  shows that this construction is indeed very important and tightly related to the mod $p$ local Lan{corollary}glands program.

  To state our main result  we need some notations. 
Let $F$ be a totally real field and $D$ a quaternion algebra with center $F$ which splits at exactly one infinite place. One can associate to $D$ a system of Shimura curves $(X_U)_U$ indexed by open compact subgroups of $(D\otimes _{\Q}\A_{f})^{\times}$ (where $\A_f$ denotes the ring of finite ad\`eles of $\Q$), which are projective and smooth over $F$. Put
\[S^D(\bFp):=\varinjlim H^1_{\textrm{\'et}}(X_{U,\overline{\Q}},\bFp)\]
where the inductive limit is taken over all open compact subgroups of $(D\otimes_{\Q}\A_f)^{\times}$. 
 Let $\brho:\Gal(\overline{\Q}/F)\ra \GL_2(\bFp)$ be an irreducible, continuous, totally odd representation.  Assume moreover $\brho$ is modular, which means that $$\pi^D(\brho):=\Hom_{\Gal(\overline{\Q}/F)}(\brho,S^D(\bFp))$$ is non-zero.  A conjecture of Buzzard, Diamond and Jarvis \cite{BDJ} says that the space $\pi^D(\brho)$ decomposes as a restricted tensor product 
 \[\pi^D(\brho) \cong {\bigotimes}_{w}'\pi_{w}\]
 where each factor $\pi_{w}$ is an admissible smooth representation of $(D\otimes_FF_w)^{\times}$ and  depends only on the restriction of $\brho$ at $w$. Note that, when $w|p$, the local factor $\pi_w$ is supposed to be the right representation in the mod $p$ local Langlands (or Jacquet-Langlands) program, and many important properties about it have been proved, see e.g. \cite{Gee11}, \cite{BD11}, \cite{EGS}.

\medskip

In this article we prove some extra property about $\pi_w$, $w|p$, when the restriction of $\brho$ at $w$ is reducible indecomposable and generic (see \S\ref{section-local}), and when $F$ is unramified at $w$. In this case, it is hoped that $\pi_w$ has a filtration of length $f:=[F_w:\Q_p]$ of the form (where $(\pi_i)_i$ denote the graded pieces of the filtration)
\begin{equation}\label{equation-intro}\pi_0\ \ligne\ \pi_1\ \ligne\ \cdots\ \ligne\ \pi_f\end{equation}
such that $\pi_i$ is a principal series if $i\in\{0,f\}$ and supersingular otherwise.
Our main theorem is as follows.
\begin{theorem} \label{theorem-main}
Assume the decomposition $\pi^D(\brho)\cong \otimes'_w\pi_w$ holds. Let $w$ be a place above $p$.  Assume that $D$ splits at $w$, and $\brho_w$ is reducible  indecomposable and generic. Assume that  $F$ is unramified at $w$ and  $[F_w:\Q_p]\geq 2$. Then $\pi_{w}$ contains a unique sub-representation $\pi$ which is of length 2 and which fits into an exact sequence
\[0\ra \pi_0\ra \pi\ra \pi_1\ra0\]
such that $\pi_0$ is a principal series and $\pi_1$ is a supersingular representation. 
\end{theorem}
In other words, we prove that $\pi_w$ contains the first two graded pieces in (\ref{equation-intro}). The main observation in the proof of Theorem \ref{theorem-main} is a theorem of Emerton \cite{Em3}, which allows us to control the ordinary part of $\pi_w$.   

Back to the local situation,  we get the following consequence.

\begin{corollary}\label{cor-intro}
Let $L$ be a finite unramified extension of $\Q_p$ of degree $f\geq 2$. There exist  a principal series $\pi_0$  and a supersingular representation $\pi_1$ such that 
\[\Ext^1_{G}(\pi_1,\pi_0)\neq0.\] 
If moreover $[L:\Q_p]=2$, and if $\pi_2$ denotes the $f$-th smooth dual of $\pi_0$ (see \cite[Prop. 5.4]{Ko}), then   
\[  \Ext^1_G(\pi_2,\pi_1)\neq0.\]
\end{corollary}

The existence of extensions in Corollary \ref{cor-intro} is not known before. Remark that such extensions do not exist when $L=\Q_p$ (\cite{Pa10}).  
We also prove, in \S\ref{section-appendix}, that when $L$ is a local field of characteristic $p$ there is no non-trivial extension of a principal series by a supersingular representation.

\medskip

The paper is organized as follows. In \S\ref{section-preparation} we prove some necessary results about smooth $\bFp$-representations of $G$, especially the structure of (irreducible) principal series. In \S3, we recall the construction of  \cite{BP} and show, under certain assumption, that the representations they constructed behave well. In \S4, we show that this assumption is satisfied in certain global situation.
In the appendix (\S\ref{section-appendix}), we show the counterpart of Theorem \ref{theorem-main} is false if $L$ is of characteristic $p$.

\section{Representation theoretic preparation}\label{section-preparation}
In this section, $L$ denotes a finite extension of $\Q_p$, with ring of integers $\cO_L$, maximal ideal $\p_L$, and residue field (identified with) $\F_q=\F_{p^f}$.   Fix a uniformiser $\varpi$ of $L$; we take $\varpi=p$ when $L$ is unramified over $\Q_p$. For $\lambda\in\F_q$, $[\lambda]\in\cO_L$ denotes its Teichm\"uller lifting.

 Let  $G=\GL_2(L)$, and define the following subgroups of $G$:
\[P=\matr{*}*0*, \ \ \overline{P}=\matr*0**,\ \ N=\matr1*01,\ \ T=\matr*00*,\]
and let $Z$ be the center of $G$.

Let $K=\GL_2(\cO_L)$ and $K_1$ be the kernel of the reduction morphism $K\twoheadrightarrow \GL_2(\F_q)$. 
Let $I\subset K$ denote the (upper) Iwahori subgroup and $I_1\subset I$ be the pro-$p$-Iwahori subgroup. Let $N_0=N\cap K$, $T_0=T\cap K$ and $T_1=T\cap K_1$.

We call a weight an irreducible representation of $K$ over $\bFp$, which is always an inflation of an irreducible representation of $\GL_2(\F_q)$. A weight is  isomorphic to (\cite[Prop. 1]{BL})
\[\Sym^{r_0}\bFp^2\otimes_{\bFp}(\Sym^{r_1}\bFp^2)^{\mathrm{Fr}}\otimes\cdots\otimes_{\bFp}(\Sym^{r_{f-1}}\bFp^2)^{\mathrm{Fr}^{f-1}}\otimes_{\bFp}{\det}^a\]
where  $0\leq r_i\leq q-1$, $0\leq a\leq q-2$ and $\mathrm{Fr}: \smatr{a}bcd\mapsto \smatr{a^p}{b^p}{c^p}{d^p}$ is the Frobenius on $\GL_2(\F_q)$.  We denote this representation by $(r_0,...,r_{f-1})\otimes {\det}^a$.

\subsection{Principal series} \label{subsection-PS}
By the work of Barthel-Livn\'e \cite{BL}, smooth  irreducible $\bFp$-representations\footnote{Although we choose to work with $\bFp$-representations in this section, all the results hold true if we replace $\bFp$ by a sufficiently large finite extension of $\F_p$.} of $G$ with a central character fall into four classes:
\begin{enumerate}
\item[(i)] one-dimensional representations, i.e. characters.

\item[(ii)]  principal series $\Ind_P^G\chi$, with $\chi\neq \chi^s$ where $\chi$ is a smooth character of $T$ inflated to $P$ and $\chi^s$ is the character of $T$ defined as $\chi^s(\smatr a00d)=\chi(\smatr d00a)$.
\item[(iii)] special series, i.e. twists of the Steinberg representation $\Sp$.

\item[(iv)] supersingular representations.
\end{enumerate} 
In this paper, only the second and fourth classes will be involved.  When we talk about a principal series $\Ind_P^G\chi$, we implicitly mean $\chi\neq \chi^s$.   Recall first the following result from \cite[\S8]{BP}.

\begin{proposition} \label{prop-extension-PS}
 Let $\pi$ be a principal series of $G$.

(i)   We have an isomorphism of  $\bFp$-vector spaces $\Hom(L^{\times},\bFp)\cong\Ext^1_{G,\zeta}(\pi,\pi)$.\footnote{ $\Ext^1_{G,\zeta}$ means that we consider extensions with a central character and $\zeta$  is the central character of $\pi$.}     

(ii) Assume $L\neq \Q_p$. Let  $\pi'$ be a smooth irreducible non-supersingular representation of $G$, then $\Ext^1_G(\pi',\pi)=0$ except when $\pi'\cong \pi$. 
\end{proposition}
\begin{proof}
 (i) follows from \cite[Cor. 8.2(ii)]{BP}, noting that $\pi$ is irreducible by our convention, hence the two cases excluded in \emph{loc. cit.} can not happen. (ii) follows from \cite[Thm.8.1]{BP} using \cite[Thm. 7.16(i)]{BP}.
\end{proof} 

We recall the construction of the isomorphism in Proposition \ref{prop-extension-PS}(i).  Let $\delta\in \Hom(L^{\times},\bFp)$. We lift $\delta$ to a homomorphism of $P$ to $\bFp$ via $P\twoheadrightarrow L^{\times}$ given by $\smatr{a}b0d\mapsto ad^{-1}$. Write $\pi=\Ind_P^G\chi$ and let $\epsilon_{\delta}$ be the extension 
\begin{equation}\label{equation-delta}
 0\ra \chi\ra \epsilon_{\delta}\ra\chi\ra0\end{equation}
corresponding to $\delta$. Explicitly, $\epsilon_{\delta}$ has a basis $\{v,w\}$ with the action of $P$ represented by $\smatr{\chi}{\chi\delta}0\chi$: $b\cdot v=\chi(b)v$ and $b\cdot w=\chi(b)w+\chi(b)\delta(b)v$ for $b\in P$. Then, inducing  to $G$ we obtain an exact sequence
\begin{equation}\label{equation-Ind(delta)}0\ra \Ind_P^G\chi\ra \Ind_P^G\epsilon_{\delta}\ra\Ind_P^G\chi\ra0\end{equation}
which is the element in $\Ext^1_{G}(\pi,\pi)$ corresponding to $\delta$. Since by definition $\delta$ is trivial on $Z$, the representation $\Ind_{P}^G\epsilon_{\delta}$ has a central character.

Taking $K_1$-invariants, (\ref{equation-Ind(delta)}) induces a long exact sequence
\[0\ra (\Ind_P^G\chi)^{K_1}\ra (\Ind_P^G\epsilon_{\delta})^{K_1}\ra (\Ind_P^G\chi)^{K_1}\overset{\partial_{\delta}}{\ra} H^1(K_1/Z_1,\Ind_P^G\chi).\]

\begin{proposition}\label{prop-partial_delta}
Notations are as above. The map $\partial_{\delta}$ is zero if and only if $\delta$ is an unramified homomorphism, i.e. $\delta$ is trivial on $\cO_L^{\times}$. Moreover, if $\partial_\delta$ is nonzero, then it is an injection.
\end{proposition}
\begin{proof}
First, if $\delta$ is unramified, then it splits when restricted to $T\cap K$, hence the sequence (\ref{equation-Ind(delta)}) also splits when restricted to $K$. This shows that $\partial_{\delta}=0$ in this case.

Now assume $\delta$ is ramified. It suffices to show that the inclusion $(\Ind_P^G\chi)^{K_1}\hookrightarrow (\Ind_P^G\epsilon_{\delta})^{K_1}$ is an equality.  By definition, $\Ind_{P}^G\epsilon_{\delta}$ identifies with the vector space of smooth functions $f:G\ra \bFp v\oplus\bFp w$ such that $f(pg)=p\cdot f(g)$, where $p\in P$, $g\in G$, and $\{v,w\}$ is the chosen basis of $\epsilon_{\delta}$ as above; the action of $G$ on $\Ind_{P}^G\epsilon_{\delta}$ is given by $(g'\cdot f)(g):=f(gg')$. Let $f$ be such a function which is fixed by $K_1$. We need to show that $f(g)\in \bFp v$ for any $g\in G$. Using the decomposition 
\[G=PK=P\matr0110K_1\coprod \Bigl(\coprod_{\lambda\in\F_q} P \matr10{[\lambda]}1 K_1\Bigr),\]
it suffices to check this for $g=\smatr0110$ and $g=\smatr10{[\lambda]}1$, $\lambda\in\F_q$. If $h=\smatr{1+\varpi a}001\in T_1$ and $g=\smatr10{[\lambda]}1$, we have (as $h\cdot f=f$):
\[ f(g)=(h\cdot f)(g)=f\Bigl(\matr{1+\varpi a}0{[\lambda](1+\varpi a)}1\Bigr) 
=h\cdot \Bigl[f\Bigl(\matr{1}0{[\lambda](1+\varpi a)}1\Bigr)\Bigr]
=h\cdot [f(g)] \]
where the last equality holds because $f$ is fixed by $K_1$. This shows that the vector $f(g)$ is fixed by $T_1$, hence lies in $\bFp v$ because $\delta$ is non-trivial on $T_1$. Similar (and simpler) argument works for $g=\smatr0110$. This finishes the proof.
\end{proof}

Since $K_1$ is normal in $K$, the space $H^1(K_1/Z_1,\pi)$ can be viewed naturally as a $K$-representation. 
\begin{corollary}\label{cor-f-dim}
Let $d:=\dim_{\bFp}\Hom(1+\p_L,\bFp)$. There is a $K$-equivariant embedding $(\pi^{K_1})^{\oplus d}\hookrightarrow H^1(K_1/Z_1,\pi)$.
\end{corollary}
\begin{proof}
The association $\delta\mapsto\partial_{\delta}$ defines an $\bFp$-linear  map 
\[\Hom(F^{\times},\bFp)\ra \Hom_K(\pi^{K_1},H^1(K_1/Z_1,\pi)).\]
By Proposition \ref{prop-partial_delta}, its kernel is the subspace of unramified homomorphisms, so we obtain an exact sequence
\[0\ra \Hom^{\rm un}(L^{\times},\bFp)\ra \Hom(L^{\times},\bFp)\ra \Hom_K(\pi^{K_1},H^1(K_1/Z_1,\pi))\]
(where $\Hom^{\rm un}$ means unramified homomorphisms).
The assertion follows from Proposition \ref{prop-partial_delta} and that $\Hom(\cO_L^{\times},\bFp)\cong \Hom(1+\p_L,\bFp)$.
\end{proof}

Let $\pi$ be a principal series. For later use, we recall some basic results about its $I_1$-invariants. By \cite[Thm. 34]{BL}, $\rsoc_K\pi$ is not necessarily irreducible in general, but it contains a \emph{unique} irreducible sub-representation which is of dimension $\geq 2$, which we denote by $\sigma$. The decomposition $G=PI_1\coprod P\Pi I_1$ (where $\Pi:=\smatr01{\varpi}0$) implies that $\pi^{I_1}$ is always 2-dimensional, spanned over $\bFp$ by $f_1,f_2$ characterized as follows:
\begin{equation}\label{equation-f1f2}f_1(\id)=1,\ f_1(\Pi)=0; \ \ \ f_2(\id)=0,\ f_2(\Pi)=1.\end{equation}
Here $\id$ denotes the identity matrix of $G$. It is easy to see that $\langle K.f_2\rangle=\sigma$ and $\sum_{\lambda\in\F_q}\smatr{\varpi}{[\lambda]}01f_2=\chi(\smatr100{\varpi})f_2$. Extending $\sigma$ to be a $KZ$-representation by letting $\smatr{\varpi}00{\varpi}$ act trivially and setting $\lambda:=\chi(\smatr100{\varpi})$, we get $\pi\cong \cInd_{KZ}^G\sigma/(T-\lambda)\otimes \chi'\circ\det$,  for some  (uniquely determined) character $\chi':L^{\times}\ra \bFp^{\times}$. Here we have used the formula (\ref{equation-Tv}) (see \S\ref{section-appendix} below) for the action of $T$.   

\subsection{Computation of $H^1(K_1/Z_1,\pi)$} 

 Assume in this subsection that $L$ is \emph{unramified} over $\Q_p$ of degree $f$. Although it is not always needed, we assume for convenience   \emph{$f\geq 2$}.

Define the following elements in the completed Iwasawa algebra $\bFp[[N_0]]$:
\[X_i:=\sum_{\mu\in \F_q}\mu^{-p^i}\matr{1}{[\mu]}01\in \bFp[[N_0]].\]
A similar proof as that of \cite[Prop. 2.13]{Sc} shows that $\bFp[[N_0]]\cong \bFp[[X_0,...,X_{f-1}]]$.  Let $\tau=(s_0,...,s_{f-1})\otimes{\det}^b$ be a weight. We can view $\tau$ as an $\bFp[[N_0]]$-module. A direct generalization of \cite[Prop. 2.14]{Sc}, using \cite[Lem. 2]{BL}, shows that  $\tau$ is isomorphic to $\bFp[[X_0,...,X_{f-1}]]/(X_0^{s_0+1},...,X_{f-1}^{s_{f-1}+1})$ as a module over $\bFp[[N_0]]\cong \bFp[[X_0,...,X_{f-1}]]$. Precisely, if $w\in \smatr0110\cdot\tau^{N_0}$ is non-zero (such a vector is unique up to a scalar), then $\tau$ is generated by $w$ as an $\bFp[[N_0]]$-module.
\begin{lemma}\label{lemma-H1(N0)}
 The $\bFp$-vector space $\Ext^1_{N_0}(\tau,\bFp)$ is of dimension $f$.
\end{lemma}
\begin{proof}
 We have the identification \[ \Ext^1_{N_0}(\tau,\bFp)\cong \Ext^1_{\bFp[[N_0]]}(M,\bFp),\]
where $M$ denotes $\bFp[[X_0,...,X_{f-1}]]/(X_0^{s_0+1},...,X_{f-1}^{s_{f-1}+1})$. The assertion then follows from the theory of Koszul complex (\cite[\S1.6]{BH}). Explicitly, for each $0\leq j\leq f-1$, the extension 
\[0\ra \bFp e_j\ra \bFp[[X_0,...,X_{f-1}]]/(X_0^{r_0+1},...,X_j^{r_j+2},...,X_{f-1}^{r_{f-1}+1})\ra M\ra0,\]
is non split (where $e_j$ is sent to $X_j^{r_j+1}$), and they form a basis of $\Ext^1_{\bFp[[N_0]]}(M,\bFp)$.
\end{proof}

We need take into account of the action of $\cH$, where $$\cH:=\bigl\{\smatr{[\lambda]}00{[\mu]}: \lambda,\mu\in\F_q^{\times}\bigr\}\subset K.$$ Note that the order of $\cH$ is prime to $p$ and $\cH$ normalizes $N_0$. 
\begin{proposition}\label{prop-cH-action}
Let $\tau=(s_0,...,s_{f-1})\otimes{\det}^b$ be as above. Let $\psi$ be a character of $\cH$ such that $\Ext^1_{\cH N_0}(\tau,\psi)\neq 0$. Then there exists $0\leq j\leq f-1$ such that $\psi=\psi_{\tau}^s\alpha^{(s_j+1)p^j}$, where $\psi_{\tau}$ is the character corresponding to the action of $\cH$ on $\tau^{N_0}$. Moreover, $\Ext^1_{\cH N_0}(\tau,\psi)$ is of dimension 1.
\end{proposition}
\begin{proof}
It is easy to see that if $V$ is an $\bFp[[\cH N_0]]$-module, and $w\in V$ is an $\cH$-eigenvector of character $\chi$, then $X_jw$ is also an $\cH$-eigenvector, but of character $\chi\alpha^{p^j}$.  So by the proof of Lemma \ref{lemma-H1(N0)}, the characters $\psi$ of $\cH$ such that $\Ext^1_{\cH N_0}(\tau,\psi)\neq0$ are the characters $\{\psi_{\tau}^s\alpha^{(s_j+1)p^j}, 0\leq j\leq f-1\}$. \end{proof}

From now on, we assume $\pi=\Ind_P^G\chi$ is a principal series satisfying: \medskip

(\textbf{H})  the $K$-socle of $\pi$  is irreducible, and if $(r_0,...,r_{f-1})\otimes{\det}^a$ is the socle, then $0\leq r_i\leq p-2$. \medskip

\begin{remark}
The first condition of ({\bf H}) amounts to demand that if we write $\chi=\eta_1\otimes\eta_2$ for $\eta_i:L^{\times}\ra \bFp^{\times}$, then $\eta_1\eta_2^{-1}$ is a ramified character.
\end{remark}
\medskip

  Define a set of weights, depending on (the $K$-socle of) $\pi$, as follows: for $0\leq j\leq f-1$, let $\sigma_j(\pi):=(s_0,...,s_{f-1})\otimes{\det}^b$ where $s_i=r_i$ for $i\notin\{j-1,j\}$ and $s_{j-1}=p-2-r_{j-1}$, $s_j=r_j+1$, $b\equiv a+p^{j-1}(r_{j-1}+1)-p^j\ (\mathrm{mod}\ q-1)$. They are well defined under the condition ({\bf H}) and the assumption $f\geq 2$.


\begin{proposition}\label{Prop-K-extension}
Let $\pi=\Ind_{P}^G\chi$ be a principal series satisfying ({\bf H}) and denote $\sigma=\rsoc_K\pi=(r_0,...,r_{f-1})\otimes{\det}^a$. Let $\tau$ be a weight. Then $\Ext^1_{K/Z_1}(\tau,\pi|_K)\neq0$ if and only if one of the following holds  
\begin{enumerate}
\item[(i)] $\tau\cong \sigma_j(\pi)$ for some $0\leq j\leq f-1$; in this case   $\dim \Ext^1_{K}(\sigma_j(\pi),\pi|_K)=1$.

\item[(ii)] $\tau\cong \sigma$; in this case $\dim \Ext^1_{K/Z_1}(\sigma,\pi|_K)=f$.
\end{enumerate} 
\end{proposition}
\begin{proof}
Iwasawa decomposition implies that $(\Ind_P^G\chi)|_{K}\cong\Ind_{P\cap K}^K(\chi|_{T\cap K})$. To simplify the notation, we  write $\psi=\chi|_{P\cap K}$ for its restriction to $P\cap K$.
By Shapiro's lemma, we have an isomorphism $\Ext^1_K(\tau,\pi|_K)\cong \Ext^1_{P\cap K}(\tau,\psi)$. Consider a non-split extension of ${P\cap K}$-representations
\begin{equation}\label{equation-extension-V}0\ra \psi\ra V\ra \tau\ra 0.\end{equation}

First assume  $V$ remains non-split when restricted to $\cH N_0$,  so that $\Ext^1_{\cH N_0}(\tau,\psi)\neq 0$. If we write $\tau=(s_0,...,s_{f-1})\otimes{\det}^b$, then Proposition \ref{prop-cH-action} implies that $\psi=\psi_{\tau}^s\alpha^{(s_j+1)p^j}$ for some $0\leq j\leq f-1$. Using the relation  $\psi^s=(r_0,...,r_{f-1})\otimes{\det}^a$, a simple calculation shows that $\psi_{\tau}=\psi_{\sigma_{j+1}(\pi)}$, hence $\tau\cong \sigma_{j+1}(\pi)$. That is we are in case (i) of the theorem. Again by Proposition \ref{prop-cH-action}, $\Ext^1_{P\cap K}(\tau,\psi)$ has dimension $\leq 1$, so the same holds for $\Ext^1_{K}(\tau,\pi|_K)$. To conclude in this case, it suffices to construct a non-zero element in $\Ext^1_{K}(\sigma_{j+1}(\pi),\pi|_K)$. In fact, \cite[Cor.5.6]{BP}   says that $\Ext^1_K(\sigma_{j+1}(\pi),\sigma)$ is non-zero and has dimension 1. In view of the exact sequence
\[\Hom_K(\sigma_{j+1}(\pi),\pi/\sigma)\ra \Ext^1_K(\sigma_{j+1}(\pi),\sigma)\ra \Ext^1_K(\sigma_{j+1}(\pi),\pi|_{K})\] 
we are reduced to show $\Hom_K(\sigma_{j+1}(\pi),\pi/\sigma)=0$. If the later space were non-zero,  we would get an inclusion $\Sigma\hookrightarrow \pi$, where  $\Sigma$ denotes the unique non-split extension of $\sigma_{j+1}(\pi)$ by $\sigma$. But $K_1$ acts trivially on $\Sigma$ (see \cite[Cor. 5.6]{BP}), we would get an inclusion $\Sigma\hookrightarrow \pi^{K_1}\cong \Ind_{P(\F_q)}^{\GL_2(\F_q)}\psi$, which contradicts \cite[Thm. 2.4]{BP}. \medskip

Now assume  $V$ is split when restricted to $\cH N_0$ and choose a $\cH N_0$-splitting $s:\tau\hookrightarrow V$. This implies that  $V$ is fixed by $\smatr{1}{\p_L}01$ since both $\psi$ and $\tau$ are. If $n\in N_0$ and $h\in T_1$, a simple calculation shows that $hn=n'nh$, for some $n'\in \smatr 1{\p_L}01$, therefore the actions on $V$ of $T_1$ and $N_0$ commute.
 \medskip

\textbf{Claim}: \textit{if $x\in \tau$ lies in the radical of $\tau$ (as $N_0$-representation), then $s(x)\in V$ is fixed by $T_1$. }\medskip

\textit{Proof of Claim.} Let $w\in \smatr0110\tau^{N_0}$ be a non-zero vector. We have seen that $w$ generates $\tau$ as an $N_0$-representation.  The condition that $x$ lies in the radical of $\tau$ is equivalent to that there exists a finite set of elements $n_i\in N_0$ such that $x=\sum_i(n_i-1)w$. Let $x$ be such an element and assume   there exists $h\in T_1$ with $(h-1)s(x)\neq 0$. The remark above implies that
\begin{equation}\label{equation-hn=nh} (h-1)s(x)=\sum_i(h-1)(n_i-1)s(w)=\sum_i(n_i-1)(h-1)s(w).\end{equation}
In particular, $(h-1)s(w)$ is non-zero. But, this vector  lies in the underlying space of $\psi$ on which $N_0$ acts trivially, so the equality (\ref{equation-hn=nh}) forces that $(h-1)s(x)=0$ as $n_i\in N_0$, contradiction. The claim follows.\medskip

 By the claim, the extension (\ref{equation-extension-V}) is the pullback of a (non-split) exact sequence  
 \[0\ra \psi\ra W\ra \tau/\mathrm{rad}(\tau)\ra0\]
on which $N_0$ acts trivially but $T_1$ acts non trivially.  This forces that $\psi=\psi_{\tau}^s$, so that $\tau\cong\sigma$ (since $\psi^s=\chi_{\sigma}$) and we are in case (ii) of the theorem. Moreover, because $T_1/Z_1\cong 1+\p_L\cong \cO_L\cong \Z_p^f$ (since $L$ is unramified) , the space $\Ext^1_{T_1/Z_1}(\psi,\psi)$ has dimension $f$.

To conclude in this case, we need show $\dim \Ext^1_{K/Z_1}(\sigma,\pi|_K)$ has dimension $\geq f$, but  it follows from Corollary \ref{cor-f-dim}.
\end{proof} 
\begin{remark}
When $L=\Q_p$, the dimension part of Proposition \ref{Prop-K-extension}(ii) is not always true, cf. \cite[Thm. 7.16(iii)]{BP}. 
\end{remark}

The above proof has the following consequence.
 \begin{corollary}\label{cor-Sigma}
Let $\pi$ be as in Proposition \ref{Prop-K-extension} and $\tau=\sigma_j(\pi)$ for some $0\leq j\leq f-1$. Let $0\ra \pi\ra E\ra \tau\ra0$ be the unique non-split $K$-extension. Then the induced sequence $0\ra \pi^{I_1}\ra E^{I_1}\ra \tau^{I_1}\ra0$ is exact.
 \end{corollary}
\begin{proof}
With notations in the proof of Proposition \ref{Prop-K-extension}, the extension $E$ comes from the extension $\Sigma$ of $\tau$ by $\sigma$. The result follows from the corresponding statement for $\Sigma$, see \cite[Prop. 4.13]{BP}.
\end{proof}

 The next lemma will be used in the proof of Proposition \ref{prop-extension-typeA}.

\begin{lemma}\label{lemma-converse}
Let $\pi$ be a principal series satisfying (\textbf{H}) and $\sigma$ be its $K$-socle. Assume that $V$ is a smooth representation of $G$ such that $\pi\hookrightarrow V$. Assume that $\Hom_K(\sigma,V|_K)$ is 1-dimensional and $\Hom_K(\sigma,V/\pi)\neq0$. Then  $V$ contains a sub-$G$-representation $V'$ which is a non-split extension of $\pi$ by $\pi$.
\end{lemma}
\begin{proof}
The assumption $\dim\Hom_K(\sigma,V)=1$ means that   $\sigma$ appears in $\rsoc_KV$ with multiplicity one and is contained in $\pi$. Therefore $V$ contains a sub-$K$-representation $E$ which fits into a non-split extension
\[0\ra \pi|_K\ra E\ra \sigma\ra0.\]

We first describe the extension $E$ more explicitly. Proposition \ref{prop-partial_delta} implies a surjective morphism $\Ext^1_{G,\zeta}(\pi,\pi)\twoheadrightarrow \Ext^1_{K,\zeta}(\sigma,\pi|_{K})$, given by pullback via $\sigma\hookrightarrow \pi|_K$. Choose an extension   
\[0\ra \pi\ra \Ind_{P}^G\epsilon_{\delta}\ra \pi\ra0\]
which lifts $E$. 
Choose a basis $\{v,w\}$ of $\epsilon_{\delta}$ as in \S\ref{subsection-PS} and define two elements of $\Ind_P^G\epsilon_{\delta}$ as follows: let $f_v\in \pi\subset \Ind_{P}^G\epsilon_{\delta}$ be the vector $f_2$   defined in (\ref{equation-f1f2}), and  $ f_w$ be the element characterized by  (write $\Pi=\smatr01p0$): 
\[\mathrm{Supp}(f_w)=P\Pi I_1=P\Pi N_0, \ \ f_w(b\Pi n)=b\cdot w,\ \forall b\in P, \forall n\in N_0.\]
Here $\cdot$ means the action of $P$ on $\epsilon_{\delta}$. 
It is easy to see that $f_w$ is well defined and fixed by $N_0$. If $h\in T_1$, then
\[hf_w=f_w+\delta(\Pi h\Pi^{-1})f_v.\] 
Moreover, the image of $f_w$ in (the quotient) $\pi$ lies in $\sigma=\rsoc_K\pi$, see \S\ref{subsection-PS}; 
that is, $f_w$ itself lies in $E$. 

Now view $f_w$ as a vector in $V$ via the inclusion $E\subset V$. Consider the operator $S:=\sum_{\lambda\in\F_q}\smatr{p}{[\lambda]}01\in\bFp[G]$. Because $f_w$ is fixed by $N_0$, we get by Lemma \ref{lemma-S2}(iii),  \[h(S(f_w))=S(h(f_w))=Sf_w+\delta(\Pi h\Pi^{-1})Sf_v=Sf_w+\delta(\Pi h\Pi^{-1})\lambda f_v.\]
Hence $f_w':=S(f_w)-\lambda f_w$ is fixed by $T_1$. By Lemma \ref{lemma-S2}(iv), for $n$ large enough,  $S^nf_w'$ is fixed by $I_1$. Moreover, by the proof of \cite[Lem. 4.1]{Pa07},  $S^{n+1}f_w'$ generates an  irreducible $K$-representation  which is isomorphic to $\sigma$. Since $\sigma$ appears with multiplicity one in the $K$-socle of $V$ by assumption, we deduce that  $S^{n+1}f_w'\in\bFp f_v$. In particular, $S^{n+2}f_w=\lambda S^{n+1}f_w$ in $V/\pi$, showing that $S^{n+1}f_w$ (which is non-zero) generates a principal series in $V/\pi$, which must be isomorphic to $\pi$.
\end{proof}

\subsection{Ordinary part}
In this subsection $L$ is a finite extension of $\Q_p$ of degree $n$. \medskip

  Recall that Emerton has defined a functor, called ordinary parts and denoted by $\Ord_P$, from the category of admissible smooth $\bFp$-representations of $G$  to the category of admissible smooth $\bFp$-representations of $T$. Let $\R^i\Ord_P$ be its right derived functors for $i\geq 1$. It follows from \cite[Prop. 3.6.1]{Em2} and \cite{EP} that $\R^i\Ord_P$ vanishes  for $i\geq n+1$.

Write $G_L=\Gal(\bQp/L)$.  Let $\epsilon:G_{L} \ra \Z_p^{\times}$ be $p$-adic the cyclotomic character ad $\omega$ be its reduction modulo $p$. View them as characters of $L^{\times}$ via the local Artin map normalized  in such a way that uniformizers of $L$ are sent to geometric Frobenii.  Denote by  $\alpha$ the character $\omega\otimes\omega^{-1}: T\ra \F_p^{\times}$.

\begin{proposition}\label{prop-Ord}
(i) If $U$ is an admissible smooth representation of $G$ and $V$ is a smooth representation of $G$, then
\[\Hom_G(\Ind_{\overline{P}}^GU,V)\cong \Hom_T(U,\Ord_P(V)).\]

(ii) There is a canonical isomorphism $\R^n\Ord_P(V)\cong V_N\otimes \alpha^{-1}$, where $\pi_N$ is the space of coinvariants (i.e. the usual Jacquet module of $V$ with respect to $P$).

(iii) We have $\Ord_P(\Ind_{ {P}}^GU)\cong U^s$ and $\R^n\Ord_P(\Ind_{{P}}^GU)\cong U\otimes\alpha^{-1}$.

(iv) If $\pi$ is an absolutely irreducible supersingular representation of $G$ over $\bFp$, then $\Ord_P(\pi)=\R^{n}\Ord_P(\pi)=0$.
\end{proposition}
\begin{proof}
(i) is \cite[Theorem 4.4.6]{Em1}. (ii) is \cite[Prop.3.6.2]{Em2}. (iii) follows from \cite[Cor. 4.3.5]{Em1} and \cite[Prop. 3.6.2]{Em2}, using the natural isomorphism $\Ind_{P}^GU\cong \Ind_{\overline{P}}^GU^s$. For (iv), the first assertion follows from (i); the second  follows from (ii). To see this, 
let  $\pi(N)\subset \pi$ be the subspace spanned by  vectors of the form $(n-1)v$, for all $n\in N$ and $v\in \pi$, so that $\pi_N=\pi/\pi(N)$. It is easily checked that $\pi(N)$ is stable under the group $P$ and non-zero. The result follows from the main result of \cite{Pa07}, which says that $\pi|_P$ is irreducible. 
\end{proof}

\subsection{A definition}\label{subsection-definition}
We recall a definition  due to Emerton \cite[\S3.6]{Em3}, which plays a crucial role below.

We normalize the local Artin map $L^{\times}\hookrightarrow G_L^{\rm ab}$ in such a way that uniformizers of $L$ are sent to geometric Frobenii.

Denote by $S$ the following subtorus of $T$
\[S:=\Bigl\{\matr a001| a\in L^{\times}\Bigr\}\subset T \]
so that $S\cong L^{\times}$. The composite of this isomorphism with the local Artin map defines an injection $\iota: S\hookrightarrow G_L^{\rm ab}$, and hence an anti-diagonal embedding
\begin{equation}\label{equation-anti-diagonal}S\hookrightarrow G_L^{\rm ab}\times S,\ \ s\mapsto (\iota(s),s^{-1}).\end{equation}

\begin{definition}\label{definition-ab-S}
Let $V$ be a  representation of $G_{L}\times S$.

(i) Let $V^{\rm ab}$ be  the maximal sub-object of $V$ on which $G_{L}$ acts through its maximal abelian quotient $G_L^{\rm ab}$. This is a $G_L\times S$-sub-representation of $V$.

(ii) Let $V^{\mathrm{ab}, S}$ be the subspace of $V^{\mathrm{ab}}$ consisting of $S$-fixed vectors, where $S$ acts through the anti-diagonal embedding (\ref{equation-anti-diagonal}) and the action of $G_L^{\rm ab}\times S$ on $V^{\rm ab}$. 
\end{definition}
\medskip

The space $V^{\mathrm{ab},S}$ is stable under the action of $G_L$ and, of course, this action factors through $G_L^{\rm ab}$.  
\begin{lemma}\label{lemma-GxS}
Let $V$ be a representation of $G_L\times S$. Assume that the action of $G_L$ on $V/V^{\mathrm{ab},S}$ factors through $G_{L}^{\rm ab}$. Then, for any subquotient $W$ of $V$ (as $G_L\times S$-representations), the action of $G_L$ on $W/W^{\mathrm{ab},S}$ also factors through $G_L^{\mathrm{ab}}$. 
\end{lemma}
\begin{proof}
 By definition, if $H_L$ denotes the kernel of the quotient map $G_L\twoheadrightarrow G_L^{\rm ab }$, then $V^{\rm ab}=V^{H_L}$ and $V^{\mathrm{ab},S}=(V^{H_L})^{S}$.  If $W$ is a sub-$G_L\times S$-representation of $V$, we deduce that $V^{\mathrm{ab},S}\cap W=W^{\mathrm{ab},S}$, hence a $G_L$-equivariant inclusion $W/W^{\mathrm{ab},S}\hookrightarrow V/V^{\mathrm{ab},S}$ and the result holds in this case.  If $W$ is a quotient of $V$, the result is obvious. The general case follows from this.
\end{proof}

\section{Local results}
\label{section-local} We keep notations of Section \ref{section-preparation}. Assume $L$ is \emph{unramified} over $\Q_p$ of degree $f\geq 2$.

\subsection{Construction of Breuil-Pa\v{s}k\={u}nas}\label{subsection-BP}

Let $\brho:\Gal(\bQp/L)\ra \GL_2(\bFp)$ be a continuous representation. Assume $\brho$ is reducible   and \emph{generic} in the sense of \cite[\S11]{BP}, i.e. $\brho$ is of the form
\[\brho\cong  \matr{\mathrm{nr}(\mu)\omega_f^{(r_0+1)+\cdots+p^{f-1}(r_{f-1}-1)}}{*}0{\mathrm{nr}(\mu^{-1})}\otimes\eta\]
where $\mu\in \bFp^{\times}$, $r_{i}\in\{0,...,p-3\}$ with $(r_0,...,r_{f-1})\neq (0,...,0),(p-3,...,p-3)$, $\omega_f$ is Serre's fundamental character of level $f$, and $\eta:\Gal(\bQp/L)\ra \bFp^{\times}$ is a continuous character.

 To $\brho$ is associated a set of weights, called Serre weights and denoted by $\mathscr{D}(\brho)$, as follows (see \cite[\S11]{BP} or \cite{BDJ}). First, the genericity condition on $\brho$ implies that $\brho$ is in the category of Fontaine-Lafaille \cite{FL}.   Writing down the associated Fontaine-Lafaille module, we define a subset $J_{\brho}$ of $\mathcal{S}:=\{0,...,f-1\}$ which we refer to \cite[\S2]{Hu13} for its precise definition. We recall that $J_{\brho}$ measures how far $\brho$ is from splitting, in the sense that $J_{\brho}=\mathcal{S}$ if and only if $\brho$ is split. Second, we define $\mathscr{D}(x_0,...,x_{f-1})$ to be the set of $f$-tuples $\tau=(\tau_0(x_0),...,\tau_{f-1}(x_{f-1}))$ satisfying the following conditions:
\begin{enumerate}
\item[(i)] $\tau_i(x_i)\in \{x_i,x_i+1,p-2-x_i,p-3-x_i\}$
\item[(ii)] if $\tau_i(x_i)\in \{x_i,x_i+1\}$, then $\tau_{i+1}(x_{i+1})\in\{x_{i+1},p-2-x_{i+1}\}$

\item[(iii)] if $\tau_i(x_i)\in\{p-2-x_i,p-3-x_i\}$, then $\tau_{i+1}(x_{i+1})\in\{p-3-x_{i+1},x_{i+1}+1\}$

\item[(iv)] if $\tau_i(x_i)\in\{p-3-x_i,x_i+1\}$, then $i\in J_{\brho}$
\end{enumerate}
with the conventions $x_f:=x_0$ and $\tau_f(x_{f}):=\tau_0(x_0)$.  Then  $\mathscr{D}(\brho)$ can be explicitly described as
\[\mathscr{D}(\brho)=\bigl\{(\tau_0(r_0),...,\tau_{f-1}(r_{f-1}))\otimes{\det}^{e(\tau)(r_0,...,r_{f-1})},\  \tau\in\mathscr{D}(x_0,...,x_{f-1})\bigr\},\]
where $e(\tau)(x_0,...,x_{f-1})$ is defined as in \cite[\S4]{BP}. Remark that  there are $2^{|J_{\brho}|}$ elements in $\mathscr{D}(\brho)$, and it always contains the weight $\sigma_0:=(r_0,...,r_{f-1})\otimes\eta\circ\det$. For $\sigma\in\mathscr{D}(\brho)$ which corresponds to $\tau\in \mathcal{D}(x_0,...,x_{f-1})$,  we set \[\ell(\sigma):=\mathrm{Card}\{i\in\mathcal{S}; \tau_i(x_i)\in \{p-2-x_i,p-3-x_i\}\},\]
and call it the \emph{length} of $\sigma$.

Let $D_0(\brho)$ be the maximal representation of $\GL_2(\F_q)$ such that 
\begin{enumerate}
\item[(i)] the $\GL_2(\F_q)$-socle of $D(\brho)$ is $\oplus_{\tau\in\mathscr{D}(\brho)}\tau$ 

\item[(ii)] each Serre weight $\tau\in\mathscr{D}(\brho)$ occurs exactly once in $D(\brho)$.
\end{enumerate}
Let $D_1(\brho)=D_0(\brho)^{I_1}$ with the induced action of $I$ and we choose an action of $\Pi=\smatr{0}1p0$ such that $\Pi^2 $ is the identity. 
The amalgam structure of $G$ (more precisely, of $\SL_2(L)$) then allows Breuil and Pa\v{s}k\={u}nas to  construct a family of smooth admissible representations of $G$ over $\bFp$, with  $K$-socle being $\oplus_{\sigma\in \mathscr{D}(\brho)}\sigma$.  The construction is as follows (see \cite{BP}). We first embed $K$-equivariantly $D_0(\brho)$ inside an injective envelope  $\Omega:=\rInj_{K}(\oplus_{\sigma\in \mathscr{D}(\brho)}\sigma)$. Then using the decomposition of $\Omega|_{I}$  we can give an action of $\smatr{0}1p0$ on $\Omega$ which is compatible with the one on $D_1(\brho)$ via the embedding we have chosen. In such a way, a theorem of Ihara allows us to get a smooth action of $G$ on $\Omega$ and we let $V$ be the sub-representation generated by $D_0(\brho)$. In particular, $V$ depends on the choice of the action of $\smatr{0}1p0$ on $\Omega$, and actually there are quite a lot of such choices. Finally, we twist $V$ by $\eta\circ\det$ to recover the determinant of $\brho$. Denote by $\mathcal{V}(\brho)$ the family of representations $V$ obtained in this way.

It is expected that such a $V\in \mathcal{V}(\brho)$ has length $f$, or, at least, there exists one $V\in\mathcal{V}(\brho)$ which has length $f$. Precisely, we hope that $V$ has a  filtration of length $f$ of the form (where $(\pi_i)_i$ denote the graded pieces of the   filtration)
\[\pi_0\ \ligne\ \pi_1\ \ligne\ \cdots\ \ligne\ \pi_f\]
such that $\pi_i$ is a principal series if $i\in\{0,f\}$ and supersingular otherwise.
This is the case when $\brho$ is (reducible) split; in this case, $V$ is a direct sum of $\pi_i$'s, $0\leq i\leq f$ (see \cite[\S19]{BP}). However, it is not even known whether such a $V$ has finite length  if $\brho$ is non-split, except for the case $L=\Q_p$ (cf. \cite[Conj. 2.3.7]{Em1}).

In the following, we  assume $\brho$ is \emph{non-split}.  

\begin{lemma}\label{lemma-V-socle}
For any $V\in\mathcal{V}(\brho)$, the $G$-socle of $V$ is an irreducible principal series and isomorphic to $\cInd_{KZ}^G\sigma_0/(T-\lambda)$ for some $\lambda\in \bFp^{\times}$.
\end{lemma}

\begin{proof}
Let $\pi$ be an irreducible sub-representation of $V$. Let $\sigma\in\mathscr{D}(\brho)$ be a Serre weight which is contained in $\rsoc_K\pi$. 
The proof of \cite[Thm. 15.4(ii)]{BP} shows that we can go from $D_{0,\sigma}(\brho)^{I_1}$ to $D_{0,\sigma_0}(\brho)^{I_1}$ using $\chi\mapsto \chi^s$ (with notations there). In particular, we see that $\sigma_0$ is also contained in $\pi$, hence $\langle G.\sigma_0\rangle \subset \pi$. But, by construction, $\langle G\cdot \sigma_0\rangle$ is an (irreducible) principal series, so that $\pi=\langle G\cdot\sigma_0\rangle$. The last assertion follows from \cite[Thm. 30]{BL}.
\end{proof}
\begin{remark}
In Lemma \ref{lemma-V-socle},  any  $\lambda\in \bFp^{\times}$ could happen. In fact, the construction in \cite{BP} does not take into account of the whole information of $\brho$.  For the representations arising from the cohomology of Shimura curves, $\lambda$ is uniquely determined by $\brho$ (see \cite{BD11} or \S\ref{section-global}).
\end{remark}

\begin{proposition}\label{prop-extension-typeA}
Let $V\in \mathcal{V}(\brho)$. Assume that $\Ord_P(V)$ is one-dimensional. Then $V$ contains a sub-representation $\pi$ which is of length 2 and fits into an exact sequence
\[0\ra \pi_0\ra \pi\ra \pi_1\ra0\]
such that $\pi_0$ is a principal series 
and $\pi_1$ is a supersingular representation. Moreover, $\pi_1$ is uniquely determined (by $V$) in the following two cases:

(i) either   $\mathscr{D}(\brho)=\{\sigma_0\}$, in which case $\rsoc_{K}(\pi_1)=\oplus_{\sigma\in\mathscr{D}(\brho^{\rm ss}),\ell(\sigma)=1}\sigma $; 

(ii) or $V^{I_1}=D_1(\brho)$, in which case we only have an inclusion
\begin{equation}\label{equation-soc_pi1}\bigoplus_{\sigma\in\mathscr{D}(\brho^{\rm ss}),\ell(\sigma)=1}\sigma\subseteq \rsoc_{K}(\pi_1),\end{equation} which is an equality when $f=2$.
\end{proposition}
\begin{remark}\label{remark-ell=1}
It is easy to check that the set $\{\sigma\in\mathcal{\brho^{\rm ss}}: \ell(\sigma)=1\}$ is exactly the set $\{\sigma_j(\pi_0): 0\leq j\leq f-1\}$ (well defined thanks to the genericity condition on $\brho$).\end{remark}
\begin{proof}
Lemma \ref{lemma-V-socle} implies that the $G$-socle of $V$ is a principal series; we denote it by $\pi_0$. By assumption, $\Ord_P(V)$ is one dimensional. We claim that $V/\pi_0$ does not admit principal series as a sub-representation.  In fact, if $\pi'\hookrightarrow V/\pi_0$ is a principal series, then $\Ext^1_{G}(\pi',\pi_0)\neq0$ since $\pi_0$ is the $G$-socle of $V$. But Proposition \ref{prop-extension-PS} implies that $\pi'\cong \pi_0$, and the corresponding extension is of the form (\ref{equation-Ind(delta)}) for certain $\delta\in\Hom(L^{\times},\bFp)$. The claim follows from the assumption on $\Ord_P(V)$ using Proposition \ref{prop-Ord}(iii). By \cite[Thm. 2]{Vi}$, V/\pi_0$ is still a smooth admissible representation, hence it contains at least one irreducible sub-representation; by the claim it must be supersingular. This  shows the first assertion of the proposition.

We define $\pi_1$ and determine its $K$-socle under assumptions (i) or (ii). 
First assume $\mathscr{D}(\brho)=\{\sigma_0\}$. By  construction recalled above,  $V$ sits inside a certain $\Omega$ such that $\Omega|_K$ is an injective envelope of $\sigma_0$. Proposition \ref{Prop-K-extension} then implies that 
\[\rsoc_K(\Omega/\pi_0)=\bigl(\bigoplus_{j=0}^{f-1}\sigma_0\bigr)\oplus \bigl(\bigoplus_{j=0}^{f-1}\sigma_j(\pi_0)\bigr).\]
By Lemma \ref{lemma-converse}, we deduce that $\rsoc_K(V/\pi_0)$ does not contain $\sigma_0$ so that 
\[\rsoc_K(V/\pi_0)\subseteq \bigl(\bigoplus_{j=0}^{f-1}\sigma_j(\pi_0)\bigr).\]
 We claim that $V/\pi_0$ contains a unique irreducible sub-representation. In fact, let $\pi_1\hookrightarrow V/\pi_0$ be an irreducible sub-representation and let $\sigma$ be a weight contained in the $K$-socle of $\pi_1$. Then $\sigma$ is of the form $\sigma_j(\pi_0)$ for some $0\leq j\leq f-1$. We check that all other  $\sigma_j(\pi_0)$'s are also contained in $\pi_1$ under the process $\chi\mapsto \chi^s$ (argument as in \cite[Thm. 15.4]{BP}). To do this, we may assume $j=1$ so that $\sigma=(p-2-r_0,r_1+1,r_2,...,r_{f-1})$ (up to twist). If we let $I(\sigma_0,\sigma^{[s]})$ be the unique sub-representation of $D_0(\brho)$ with cosocle $\sigma^{[s]}$ (see \cite[Cor. 3.12]{BP}), where $\sigma^{[s]}:=(r_0+1,p-2-r_1,p-1-r_2,...,p-1-r_{f-1})$ (up to a twist uniquely determined by $\sigma$) is as in \cite[page 9]{BP}. Using \cite[Cor. 4.11]{BP}, $\sigma_0(\pi_0)$ occurs in $I(\sigma_0,\sigma^{[s]})$ as an irreducible constituent. This shows that $\sigma_0(\pi)$ also occurs in the $K$-socle of $\pi$. Repeating this argument gives the result and shows that $\rsoc_K\pi_1=\oplus_{j=0}^{f-1}\sigma_{j}(\pi_0)= \oplus_{\sigma\in\mathscr{D}(\brho),\ell(\sigma)=1}\sigma$ by Remark \ref{remark-ell=1}.

Now assume (ii) $V^{I_1}=D_1(\brho)$.  
By the construction,  $V$ sits inside certain $\Omega\in \mathrm{Mod}_G^{\rm sm}$ such that $\Omega|_K$ is an injective envelope of $\oplus_{\sigma\in\mathscr{D}(\brho)}\sigma$. Since $\rsoc_K(\pi_0)=\sigma_0$, we can $K$-equivariantly  decompose $\Omega=\oplus_{\sigma\in\mathscr{D}(\brho)}\Omega_\sigma$ so that $\rsoc_K(\Omega_{\sigma})=\sigma$ for each $\sigma$ and that $\pi_0$ is contained in $\Omega_{\sigma_0}$. Therefore $\Omega/\pi_0=(\Omega_{\sigma_0}/\pi_0)\oplus (\oplus_{\sigma\neq \sigma_0}\Omega_{\sigma})$ and Proposition \ref{Prop-K-extension} then implies that 
\[\rsoc_K(V/\pi_0)\subseteq \rsoc_K(\Omega/\pi_0)=\bigl(\bigoplus_{j=0}^{f-1}\sigma_0\bigr)\oplus \bigl(\bigoplus_{j=0}^{f-1}\sigma_j(\pi_0)\bigr)\oplus \bigl(\bigoplus_{\sigma\in\mathscr{D}(\brho),\sigma\neq \sigma_0}\sigma\bigr).\]
Here, although $\sigma_j(\pi_0)$ is possibly isomorphic to some $\sigma\in \mathscr{D}(\brho)$ in view of Remark \ref{remark-ell=1} (automatically of length 1), we use $\sigma_j(\pi)$ to emphasize that it is a {sub}-representation of $\Omega_{\sigma_0}/\pi_0$ and use $\sigma\in \mathscr{D}(\brho)$ to emphasize that it is contained in $\Omega_{\sigma}$. Lemma \ref{lemma-converse} implies that $V/\pi_0$ does not admit $\sigma_0$ as a sub-$K$-representation, hence
\[\bigoplus_{\sigma\in\mathscr{D}(\brho),\sigma\neq \sigma_0}\sigma\subseteq \rsoc_K(V/\pi_0)\subseteq \bigl(\bigoplus_{j=0}^{f-1}\sigma_j(\pi_0)\bigr)\oplus \bigl(\bigoplus_{\sigma\in\mathscr{D}(\brho),\sigma\neq \sigma_0}\sigma\bigr).\] 
 Let $\sigma\in \mathscr{D}(\brho)$  such that $\ell(\sigma)=1$. We know $\sigma=\sigma_j(\pi_0)$ for some $0\leq j\leq f-1$ by Remark \ref{remark-ell=1}. Let $\Sigma$ be the unique non-split extension of $\sigma_j(\pi_0)$ by $\sigma_0$; then  $\Sigma^{I_1}=\sigma_0^{
 I_1}\oplus\sigma_j(\pi_0)^{I_1}$ is $2$-dimensional by \cite[Prop. 4.13]{BP}. We have $\Hom_K(\Sigma,V)=0$; if not, we would have $\Sigma\hookrightarrow V$ and therefore $\Sigma\oplus \sigma\hookrightarrow V$ which would contradict the assumption $V^{I_1}=D_1(\brho)$ which is multiplicity free (\cite[Cor. 13.5]{BP}). In all we get
\[\rsoc_K(V/\pi_0)\subseteq \bigl(\bigoplus_{\sigma\in\mathscr{D}(\brho^{\rm ss}),\ell(\sigma)=1}\sigma\bigr)\oplus \bigl(\bigoplus_{\sigma\in\mathscr{D}(\brho),\ell(\sigma)\geq 2}\sigma\bigr).\]
As in case (i), we show that $V/\pi_0$ admits a unique irreducible sub-$G$-representation. For this, we show that any sub-representation $\pi_1$ of $V/\pi_0$ contains $\oplus_{\sigma\in\mathscr{D}(\brho^{\rm ss}),\ell(\sigma)=1}\sigma$ in its $K$-socle. Indeed,  if $\pi_1$ contains some $\sigma\in \mathscr{D}(\brho)$ with $\ell(\sigma)\geq 2$, then   the  same argument as in the proof of \cite[Thm. 15.4]{BP}, using the process $\chi\mapsto \chi^s$,  shows that $\pi_1$ also contains another $\sigma'\in \mathscr{D}(\brho)$ with $\ell(\sigma')<\ell(\sigma)$. Remark that, although it is not necessary for us, we can guarantee that $\sigma'$ belongs to $\mathscr{D}(\brho)$, not just to $\mathscr{D}(\brho^{\rm ss})$. So we may assume $\ell(\sigma)=1$ at the beginning. Still using the process $\chi\mapsto \chi^s$, we claim that $\pi_1$ contains all other $\sigma'\in\mathscr{D}(\brho^{\rm ss})$ with $\ell(\sigma')=1$. The argument is similar to that of case (i) above but slightly different, as follows. 

Let $D_{0,0}(\brho):=D_0(\brho)\cap \pi_0$ and let $D_{0,1}(\brho)$ be the maximal sub-$K$-representation of $D_0(\brho)/D_{0,0}(\brho)$ which does not contain any element of $\mathscr{D}(\brho^{\rm ss})$ of length $\geq 2$.  By \cite[Lem. 5.1]{Hu13}, the $K$-socle of $D_{0,1}(\brho)$ is exactly $\oplus_{\sigma\in\mathscr{D}(\brho^{\rm ss}), \ell(\sigma)=1}$. By \cite[Thm. 13.1]{BP}, $D_{0,1}(\brho)$ is a sub-$K$-representation of $D_{0,1}(\brho^{\rm ss})$ (see \cite[Thm. 15.4(ii)]{BP} for this notation), since it does not contain any element of $\mathscr{D}(\brho^{\rm ss})$ of length $0$ or $\geq 2$.   Moreover, $D_{0,1}(\brho)^{I_1}$ is stable under the action of $\smatr{0}1p0$ and compatible with that of $D_{0,1}(\brho)^{I_1}$ because they are both multiplicity free. In all, we are reduced to check the claim in the case $\brho=\brho^{\rm ss}$, which is done in \cite[Thm. 1.54(ii)]{BP}.

To conclude, we just let $\pi_1$ be the sub-representation of $V/\pi_0$ generated by $\oplus_{\sigma\in\mathscr{D}(\brho^{\rm ss}),\ell(\sigma)=1}$.  The inclusion (\ref{equation-soc_pi1}) follows from the claim and the equality when $f=2$ is obvious.
\end{proof}
 \begin{remark} 
In \cite{EGS}, it is  shown that for some $V\in \mathcal{V}(\brho)$ coming from cohomology of Shimura curves, the condition (ii) of Proposition \ref{prop-extension-typeA} is verified.
Moreover, for such a $V$,  it is hoped  that (\ref{equation-soc_pi1}) is always an equality.
\end{remark}

\subsection{The case $f=2$}\label{subsection-f=2}

We have seen that non-split $G$-extensions of a supersingular representation by a principal series exist (under the conditions which will be checked in \S\ref{section-global}). In this section, we  deduce from this  the existence of non-split extensions of the converse type in the case $f=2$, namely extensions of the form
 \[0\ra \pi_1\ra *\ra \pi_2\ra 0\]
with $\pi_1$ supersingular and $\pi_2$ a principal series.
\medskip

We keep notations in the previous subsection.  
 
\begin{proposition}\label{prop-extension-typeB}
Assume $f=2$. Let $\pi_0=\Ind_{{P}}^G\chi$ and $\pi_1$ be a supersingular representation. If $\Ext^1_G(\pi_1,\pi_0)\neq0$, then $\Ext^1_G(\Ind_{{P}}^G\chi^s\alpha,\pi_1)\neq0$.
\end{proposition}

\begin{proof}
Let \begin{equation}\label{equation-extension-A}0\ra \pi_0\ra V\ra \pi_1\ra0\end{equation}
be a non-split extension.
Apply the functor $\Ord_P$  and using   Lemma \ref{prop-HOrd} below, we get a surjection
 \[ \R^1\Ord_P\pi_1\twoheadrightarrow \R^2\Ord_P\pi_0.\]
 By Proposition \ref{prop-Ord}(ii) (as $f=2$), we know $\R^2\Ord_P\pi_0\cong \chi\alpha^{-1}$,   hence $\R^1\Ord_P\pi_1$ is non-zero and admits a quotient isomorphic to $\chi\alpha^{-1}$. Because there is no non-trivial extension between two non-isomorphic $T$-characters, we deduce that $\R^1\Ord_P\pi_1$ also contains a sub-character isomorphic to $\chi\alpha^{-1}$.
The assertion follows from the long exact sequence \cite[(3.7.5)]{Em2} which implies that $\Ext^1_G(\Ind_{\overline{P}}^G\chi\alpha^{-1},\pi_1)\cong \Hom_T(\chi\alpha^{-1},\R^1\Ord_P(\pi_1))\neq0$.
\end{proof}

\begin{lemma}\label{prop-HOrd}
We have $\R^2\Ord_PV=0$.
\end{lemma}
\begin{proof}
By assumption $f=2$, Proposition \ref{prop-Ord}(iii) implies that $\R^2\Ord_PW=W_N\otimes \alpha^{-1}$ for any $G$-representation $W$.    The sequence (\ref{equation-extension-A}) induces an exact sequence $ (\pi_0)_N\ra V_N\ra (\pi_1)_N\ra0$. Since $(\pi_1)_N=0$ by Proposition \ref{prop-Ord}(iv), it suffices to prove that the map $(\pi_0)_N\ra V_N$ is zero. If not, then $V_N\cong (\pi_0)_N$ as $(\pi_0)_N$ is one dimensional, and the adjunction formula (see \cite[\S3.6]{Em1}) implies that
  \[\Hom_G(V,\Ind_P^G(\pi_0)_N)=\Hom_G(V,\pi_0)\neq 0,\]
As a consequence, the extension (\ref{equation-extension-A}) splits, giving the desired contradiction.  
\end{proof}

We have remarked that when $\brho$ is reducible, we hope that any $V\in \mathcal{V}(\brho)$ is a successive extension of the irreducible representations $(\pi_i)_{0\leq i\leq 2}$, with $\pi_0,\pi_2$ being principal series. It is easily checked that if we write $\pi_0=\Ind_{P}^G\chi$ for suitable $\chi$, then $\pi_2=\Ind_P^G\chi^s\alpha$, which is compatible with Proposition \ref{prop-extension-typeB}.

\begin{remark}\label{remark-Ko}
Note that $\pi_2$ is exactly the $f$-th smooth dual of $\pi_0$, in the sense of \cite[Def. 3.12 \& Prop. 5.4]{Ko}. 
\end{remark}

\section{Global results} \label{section-global}
We prove the main result of this article in this section. 
\medskip

 Let $F$ be a totally real number field. For each finite place $v$ of $F$, denote by $F_v$  the completion of $F$ at $v$. Write $G_{F}=\Gal(\overline{F}/F)$ and $G_{F_v}=\Gal(\overline{F_v}/F_v)$, and we identify $G_{F_v}$ with a subgroup of $G_F$ by fixing an embedding $\overline{F}\hookrightarrow \overline{F_v}$. We fix a finite extension $E$ of $\Q_p$, which serves as the coefficient field and is allowed to be enlarged. Write $\cO_E$ for the ring of integers of $E$, $k_E$ its residue field, and $\varpi_E$ a fixed uniformizer. 

\subsection{A theorem of Emerton} 

Let $D$ be a quaternion algebra over $F$ which splits  at exactly one infinite place denoted by $\tau$. Fix an isomorphism $D_{\tau}\overset{\mathrm{def}}{=}D\otimes_{F,\tau}\R\cong \mathrm{M}_2(\R)$. Let $D_{f}^{\times}=(D\otimes_{\Q}\A_f)^{\times}$. 

For any open compact subgroup $U\subset D_f^{\times}$, let $X_U$ be the projective smooth algebraic curve over $F$ associated to $U$ and consider the \'etale cohomology with coefficients in $A$ 
\[H^1_{\textrm{\'et}}(X_{U,\overline{\Q}},A),\]
where $A$ denotes one of $E$, $\cO_E$, or $\cO_E/\varpi_E^s$ for some $s>0$. For two open compact subgroups $V\subseteq U$ of $D_f^{\times}$, we have natural morphisms of algebraic curves $X_{V}\ra X_U$ defined over $F$, which induces a $\Gal(\overline{F}/F)$-equivariant map 
\[H^1_{\textrm{\'et}}(X_{U,\bar{\Q}},A)\ra H^1_{\textrm{\'et}}(X_{V,\bar{\Q}},A).\]
Define 
\[S^D(A):=\varinjlim H^1_{\textrm{\'et}}(X_{U,\bar{\Q}},A)\]
where the limit is taken over all the open compact subgroups $U\subset D_f^{\times}$. It carries a continuous action of $\Gal(\overline{F}/F)$ and a smooth admissible action of $D_{f}^{\times}$ commuting with each other.

Let $\brho:G_F\ra \GL_2(k_E)$ be an irreducible, continuous, totally odd representation. Assume that $\brho$ is modular, in the sense that $\Hom_{G_F}(\brho,S^D(k_E))\neq 0$. Let $\Sigma$ be a finite set of finite places of $F$ which contains all those places dividing $p$, or at which $U_w$ is not maximal, or  $D$ or $\brho$ is ramified. Define as usual the Hecke operators 
\[T_v=[\GL_2(\cO_{F_v})\matr{\varpi_v}001\GL_2(\cO_{F_v})],\ \ S_v=[\GL_2(\cO_{F_v})\matr{\varpi_v}00{\varpi_v}\GL_2(\cO_{F_v})]\]  
and let $\mathbb{T}^{\Sigma}(U)$ denote the commutative $A$-algebra generated by $T_v$ and $S_v$ for $v\notin \Sigma$. We let $\m_{\brho}^{\Sigma}=\m_{\brho}^{\Sigma}(U)$ denote the maximal ideal of $\mathbb{T}^{\Sigma}$ corresponding to $\brho$, i.e. satisfying
\[T_v\  \textrm{mod}\ \m_{\brho}^{\Sigma} \equiv \textrm{trace}(\brho(\mathrm{Frob}_v))\ \ \mathrm{and}\ \ \textbf{N}(v)S_v\  \textrm{mod}\ \m_{\brho}^{\Sigma}\equiv\det(\brho(\mathrm{Frob}_v))\]
for all $v\notin \Sigma$. Let 
\[S^D(U,k_E)[\m_{\brho}^{\Sigma}]=\{f\in S^D(U,k_E)|Tf=0\ \mathrm{for\ all\ }T\in \m_{\brho}^{\Sigma}\}.\]
By \cite[Lemma 4.6]{BDJ}, it is independent of $\Sigma$, so  denote it by $S^D(U,k_E)[\m_{\brho}]$. We can consider the direct limit over $U$ of the spaces $S^U(U,k_E)[\m_{\brho}]$ which yields $S^D(k_E)[\m_{\brho}]$.
Recall the following conjecture due to Buzzard, Diamond and Jarvis \cite[Conj. 4.9]{BDJ}.

\begin{conjecture}\label{conjecture-BDJ}
 The representation $S^D(k_E)[\m_{\brho}]$ of $G_F\times D_f^{\times}$ is isomorphic to a restricted tensor product \begin{equation}\label{equation-BDJ}S^D(k_E)[\m_{\brho}]\cong\brho\otimes(\otimes_{w}'\pi_{w}),\end{equation}
where $\pi_{w}$ is a smooth admissible representation of $D_{w}^{\times}$ such that 
\begin{itemize}
\item[$\bullet$] if $w$ does not divide $p$, then $\pi_{w}$ is the representation attached to $\brho_w:=\brho|_{G_{F_w}}$ by the modulo $\ell$ local Langlands and Jacquet-Langlands correspondence, see \cite[\S4]{BDJ};

\item[$\bullet$] if $w|p$, then $\pi_w\neq 0$; moreover if $F$ and $D$ are unramified at $v$, and $\sigma$ is an irreducible $k_E$-representation of $\GL_2(\cO_{F_w})$, then 
$$\Hom_{\GL_2(\cO_{F_w})}(\sigma,\pi_w)\neq0 \ \Longleftrightarrow \sigma\in W (\brho_w).$$  
Here $W(\brho_w)$ is a certain set of Serre weights associated to $\brho_w$ which, when $\brho_w$ is generic, coincides with $\mathscr{D}(\brho_w)$ (in \S\ref{section-local}) up to normalisation. 
\end{itemize}
\end{conjecture}
 
From now on,  assume that $D$ splits at some finite place $v$ lying above $p$.  For $U^v$ an open subgroup of $\prod_{w\neq v}\cO_{D_w}^{\times}$, we write
\[S^D(U^v,A):=\varinjlim S^D(U^vU_v,A) \]
where the inductive limit is taken over all compact open subgroups $U_v$ of $D_v^{\times}\cong\GL_2(F_v)$.  
The following result is due to Emerton (see Definition \ref{definition-ab-S} for the notation). 
\begin{theorem}\label{theorem-Emerton}
For any $n\geq 0$, the action of $\Gal(\bQp/F)$ on the cokernel of the embedding
\[\Ord_P\bigl(S^D(U^v,k_E)_{\m_{\brho}}\bigr)^{\mathrm{ab},S_v}\hookrightarrow \Ord_P\bigl(S^D(U^v,k_E)_{\m_{\brho}}\bigr)\]
factors through $G_{F_v}^{\rm ab}$.
\end{theorem}
\begin{proof}
This is Theorem 5.6.11 of \cite{Em3} in our setting. The proof of Emerton works equally in this case. 
In fact, Lemmas 5.6.7 and 5.6.8 \emph{loc. cit.} hold true, with the only change being to replace the absolute value $|\ |_p$ by  $|\ |_v$, the absolute value on $F_v$ normalised as $|\varpi_{v}|_{v}:=q_v^{-1}$. 
We also need 
\[S^D(U^v, \cO_E)_{\m_{\brho}}^{N_{0,v}}/\varpi_L S^D(U^v,\cO_E)_{\m_{\brho}}^{N_{0,v}}\ra S^D(U^v,k_E)_{\m_{\brho}}^{N_{0,v}}\]
to be surjective to apply Lemma 5.6.3 \emph{loc.cit.}. This is a consequence, by taking inductive limit over $r$, of the isomorphisms  (provided $r$ is large enough so that $U^vI_{r,v}$ is neat)
\[S^D(U^v,\cO_E)_{\m_{\brho}}^{I_{r,v}}/\varpi_L S^D(U^v,\cO_E)_{\m_{\brho}}^{I_{r,v}}\ra S^D(U^v,k_E)_{\m_{\brho}}^{I_{r,v}}.\]
Here,  we denote by  $I_{r,v}$ the open subgroup of $\GL_2(\cO_{F_v})$ defined as $\{g\equiv \smatr{1}{*}01\mod \varpi_v^{r}\}$.
\end{proof}

\subsection{Application}

Assuming the existence of the decomposition as in Conjecture  \ref{conjecture-BDJ}, we get the following information about the local factor $\pi_v$ (for $v|p$) when $\brho|_{G_{F_v}}$ is reducible with scalar endomorphisms. 
 
\begin{theorem}\label{thm-main}
Assume the decomposition (\ref{equation-BDJ}) holds. If for some $v|p$, $\brho_v$ is reducible indecomposable and isomorphic to $\smatr{\psi_1}{*}0{\psi_2}$, with $\psi_1\neq\psi_2$, then $\Ord_P(\pi_{D,v}(\brho))$ is an admissible semi-simple $T_v$-representation. 
\end{theorem} 
\begin{proof}
The natural inclusion $S^D(U^v,k_E)[\m_{\brho}]\hookrightarrow S^D(U^v,k_E)_{\m_{\brho}}$ induces an inclusion $\Ord_P(S^D(U^v,k_E)[\m_{\brho}])\hookrightarrow \Ord_P(S^D(U^v,k_E)_{\m_{\brho}})$, hence   Lemma \ref{lemma-GxS}
 and 
Theorem \ref{theorem-Emerton}  imply that the action of $G_{F_v}$ on the cokernel of
$$\bigl(\Ord_P(S^D(U^v,k_E)[\m_{\brho}])\bigr)^{\mathrm{ab},S_v}\hookrightarrow \Ord_P(S^D(U^v,k_E)[\m_{\brho}])$$
factors  through $G_{F_v}^{\rm ab}$. 

Assume that the decomposition (\ref{equation-BDJ}) holds. Then, as a representation of $G_{F_v}\times \GL_2(F_v)$, $S^D(U^v,k_E)[\m_{\brho}]$ is isomorphic to $(\brho_v\otimes \pi_v)^{\oplus r}$ for some integer $r\geq 1$. Apply Lemma \ref{lemma-GxS} again, we deduce that the action of $G_{F_v}$ on the cokernel of \[\Ord_P(\brho_v\otimes\pi_v)^{\mathrm{ab},S_v}\hookrightarrow \Ord_P(\brho_v\otimes \pi_v)\] factors through $G_{F_v}^{\rm ab}$.  

Now assume that $\Ord_P(\pi_{v})$ is not semi-simple. Because there is no non-trivial extension between two non-isomorphic $k_E$-characters of $T_v$, there must exist some character $\chi:T_v\ra k_E^{\times}$ and  some non-trivial extension as (\ref{equation-delta})
\[0\ra \chi\ra \epsilon_{\delta}\ra \chi\ra0\]
such that $\epsilon_{\delta}$ appears as a \emph{subquotient} of  $\Ord_P(\pi_v)$. This implies that $\Ord_P(\brho_v\otimes\pi_v)$, which is equal to $\brho_v\otimes \Ord_P(\pi_v)$, has a  $G_{F_v}\times T_v$-equivariant, hence $G_{F_v}\times S_v$-equivariant, subquotient of the form  
$\brho_v\otimes \epsilon_{\delta}$. 
Applying Lemma \ref{lemma-GxS} to it shows that the action of $G_{F_v}$ on the cokernel of 
\begin{equation}\label{equation-cokernel}
(\brho_v\otimes\epsilon_{\delta})^{\mathrm{ab},S_v}\hookrightarrow \brho_v\otimes \epsilon_{\delta}\end{equation} 
factors through $G_{F_v}^{\rm ab}$. We claim that $(\brho_v\otimes \epsilon_{\delta})^{\mathrm{ab},S_v}$ is at most $1$-dimensional over $k_E$, while $\brho_v\otimes \epsilon_{\delta}$ is $4$-dimensional, and  the cokernel of (\ref{equation-cokernel}) admits a quotient isomorphic to $\brho_v\otimes \chi$. Because  $\brho_v$ is reducible and indecomposable by assumption, the action of $G_{F_v}$ on $\brho_v\otimes\chi$ does not factor through $G_{F_v}^{\rm ab}$, hence a contradiction which shows that $\Ord_P(\pi_v)$ is semi-simple. 

To verify the claim, we choose a basis $\{v_1,v_2\}$  of $\brho_v$ over $k_E$ such that $g\cdot v_1=\psi_1(g)v_1$ for $g\in G_{F_v}$; also choose a basis $\{w_1,w_2\}$ of $\epsilon_{\delta}$ such that $s\cdot w_1=\chi(s)w_1$ and $s\cdot w_2=\chi(s)(w_2+\delta(s)w_1)$ for $s\in T_v$. It is clear that $V^{\rm ab}\subset \psi_1\otimes \epsilon_{\delta}$, as $\brho_v$ is indecomposable and $\psi_1\neq \psi_2$. Since the action of $S_v$ on $\brho_v\otimes \epsilon_{\delta}$ is via the anti-diagonal embedding $S_v\hookrightarrow G_{F_v}^{\rm ab}\times S_v$, $s\mapsto (\iota(s),s^{-1})$, we get  
\[s\cdot(v_1\otimes w_1)=\psi_1(\iota(s))v_1\otimes \chi(s^{-1})w_1\]
\[s\cdot(v_1\otimes w_2)=\psi_1(\iota(s))v_1\otimes [\chi(s^{-1})(w_2+\delta(s^{-1})w_1)].\]
Because $\delta:F_v^{\times}\cong S_v\ra k_E$ is non-trivial and the extension $\epsilon_{\delta}$ admits a central character, there exists $s\in S_v$ such that $\delta(s^{-1})=-\delta(s)\neq0$, hence 
\[(\brho_v\otimes\epsilon_{\delta})^{\mathrm{ab},S_v}\subseteq k_E (v_1\otimes w_1).\]
The claim follows easily from this. 
\end{proof}

Combined with results proved in \S\ref{section-local}, we get the following corollary.
\begin{corollary}\label{cor-main}
Keep assumptions in Theorem \ref{thm-main}. Assume moreover that $\brho_v$ is generic and $F_v$ is unramified over $\Q_p$ of degree $\geq 2$.   Then the $\GL_2(F_v)$-representation $\pi_v$ contains a sub-representation $\pi$ which is of length $2$ and fits into a non-split exact sequence
\[0\ra \pi_0\ra \pi\ra\pi_1\ra0\]
with $\pi_0$ a principal series and $\pi_1$  supersingular.  If moreover $[F_v:\Q_p]=2$, and if $\pi_2$ denotes the $f_v$-th smooth dual of $\pi_0$, then  $\Ext^1_{\GL_2(F_v)}(\pi_2,\pi_1)\neq 0$. 
\end{corollary}
\begin{proof}
Under our assumptions together with the genericity condition of $\brho_v$, we can apply \cite[Thm. 3.7.1]{BD11} to get that $\rsoc_{\GL_2(\cO_{F_v})}(\pi_v)$ is of multiplicity 1. Then Theorem \ref{thm-main} implies that $\Ord_{P_v}(\pi_v)$ is one-dimensional,  so that the result follows from Propositions \ref{prop-extension-typeA} and \ref{prop-extension-typeB} and Remark \ref{remark-Ko}.
\end{proof} 

We can get rid of the assumption (\ref{equation-BDJ}) in Corollary \ref{cor-main} as follows (note however that then $\pi_v$ does not make sense).    

\begin{proof}[Proof of Corollary \ref{cor-intro}]  
In \cite{BD11}, Breuil and Diamond explicitly constructed a certain $k_E$-representation $\pi_v'$ of $\GL_2({F_v})$, and  showed that $\pi_v'$ must be the local factor $\pi_v$ if (\ref{equation-BDJ}) holds (see \cite[Cor. 3.7.4]{BD11}). We have a $G_{F_v}\times \GL_2(F_v)$-equivariant embedding   $\brho_v\otimes \pi_v'\hookrightarrow S^D(U^v,k_E)[\m_{\brho}]$, so that $\Ord_P(\pi_v')$ is semi-simple by the same proof of Theorem \ref{thm-main}.  Moreover, under genericity assumption on $\brho_v$, \cite[Thm. 3.7.1]{BD11} says that the $K_v$-socle of $\pi_v'$ is of multiplicity one, hence $\Ord_P(\pi_v')$ is of dimension 1.  We deduce that $\pi_v'$ contains a non-split extension of $\pi_1$ by $\pi_0$. The last assertion of Corollary \ref{cor-intro} is deduced as in Corollary \ref{cor-main}.  
 \end{proof}

\section{Appendix}\label{section-appendix}

In this section, we prove, contrast to Proposition \ref{prop-extension-typeB}, that there is no non-trivial extension of principal series by supersingular representations, when $L$ is a local field of characteristic $p$.  Notations are the same as in Section \ref{section-preparation}.\medskip

 We fix a uniformizer $\varpi$ of $L$.  For an irreducible  smooth  representation $\sigma$ of $K$, we view it as a representation of $KZ$ by letting $\smatr{\varpi}00{\varpi}$ act trivially. Consider the compact induction $\cInd_{KZ}^G\sigma$ and recall that the $\bFp$-algebra $\End_G(\cInd_{KZ}^G\sigma)$ is isomorphic to $\bFp[T]$ where $T$ is a certain Hecke operator (normalized as in \cite{BL}). To describe the action of $T$, for $g\in G$ and $v\in \sigma$, denote by $[g,v]\in \cInd_{KZ}^G\sigma$ the function supported on $KZg^{-1}$ and such that $[g,v](g^{-1})=v$. 
Let $v_0\in \sigma^{I_1}$ be a non-zero vector and recall that  $\id$ denotes the identity matrix of $\GL_2(L)$,   then the action of $T$ on $\cInd_{KZ}^G\sigma$ is characterized by the formula (and by the $G$-equivariance of $T$):
\begin{equation}\label{equation-Tv}T([\id,v_0])=\sum_{\lambda\in\F_q}\Bigl[\matr{\varpi}{[\lambda]}01,v_0\Bigr]+\epsilon(\sigma)[\Pi,v_0]\end{equation}
where $\epsilon(\sigma)=1$ if $\sigma$ is of dimension $1$ and $\epsilon(\sigma)=0$ otherwise.
 
 In \cite[Def. 2.7]{Hu12} is defined an operator by the formula: 
\[S=\summ_{\lambda\in\F_q}\matr{\varpi}{[\lambda]}01\in\bFp[G]\]
which acts on any $G$-representation. By (\ref{equation-Tv}) we get immediately 
\begin{equation}\label{equation-S=T}T([\id,v_0])=S([\id,v_0])\end{equation}
when $\sigma$ is of dimension $\geq 2$. 
Set $S^1=S$ and by induction $S^n=S\circ S^{n-1}$ for $n\geq 1$.  We then have $S^{n+m}=S^n\circ S^m=S^m\circ S^n$. Recall some useful properties of $S$.

\begin{lemma}\label{lemma-S2}
Let $\pi$ be a smooth representation of $G$ and let $v\in \pi$.

(i) If $v$ is an eigenvector of $\cH$, then so is $Sv$ for the same eigencharacter.

(ii)  If $v$ is fixed by $N_0$, then so is $Sv$. 

(iii) If $v$ is fixed by $N_0$, then so is $h\cdot v$ for $h\in T_1$, and we have the formula
$h\cdot Sv=S(h\cdot v)$
for all $h\in T_1$.

 (iv) If $v$ is fixed by  $\smatr{1+\p_L}{\cO_L}{\p_L^{n+1}}{1+\p_L}$ with $n\geq 1$,  then $Sv$ is fixed by $\smatr{1+\p_L}{\cO_L}{\p_L^{n}}{1+\p_L}$. In particular, if $v$ is fixed by $I_1$, then so is $Sv$.

(v) Assume that $L$ is of characteristic $p$ and $\pi$ is a supersingular representation of $G$. Then there exists an integer $n\gg 0$ such that $S^nv=0$.
\end{lemma}
\begin{proof}
(i), (ii) and (iii) follow from an easy calculation, see \cite[Lem. 2.8]{Hu12} for the details. (iv)  is proved in \cite[\S4.2, (4.4)]{Hu12}. (v) is just \cite[Thm. 5.1]{Hu12}, which requires $L$ to be of characteristic $p$.
\end{proof}

\begin{theorem}\label{thm-super}
Assume that $L$ is of characteristic $p$. Let $\pi_1$ be a supersingular representation and $\pi_2$ be a principal series or a special series of $G$. Then $\Ext^1_G(\pi_2,\pi_1)=0$.
\end{theorem}

\begin{proof}
Consider an extension of $G$-representations
\[0\ra \pi_1\ra V\ra\pi_2\ra0.\]
Because $\pi_2$ is a principal series or a special series, it always contains a sub-$K$-representation of dimension $\geq 2$, say $\sigma$,  and $\pi_2$ is a quotient of $\cInd_{KZ}^G\sigma/(T-\lambda)\otimes \chi\circ\det$ for some $\lambda\in\bFp^{\times}$ and $\chi:L^{\times}\ra \bFp^{\times}$, see \cite[Thm. 33]{BL}.
Let $\bar{w}\in \sigma^{I_1}\subset \pi_2$ be a non-zero vector, which is unique up to a scalar and is automatically an eigenvector of $\cH$. Choose a lifting $w\in V$ of $\bar{w}$ arbitrarily. Since the order of $\cH$ is prime to $p$, we may choose $w$ so that it is an eigenvector of $\cH$. Denote by $M=\langle N_0.w\rangle\subset V$ the sub-$N_0$-representation generated by $w$ and choose vectors $v_i\in \pi_1$,    where $i$ runs over a finite set,  such that $\{w,v_i\}$ forms a basis of $M$. Then by \cite[Lem. 5.2]{Hu12}, for any $n>0$, $\langle N_0.S^nw\rangle$  is spanned by the vectors $\{S^nw,S^nv_i\}$. But, since $M$ is finite dimensional and $\pi_1$ is supersingular, Lemma \ref{lemma-S2}(v) implies that  $S^nv_i=0$  for $n\gg0$, hence $S^nw$ is fixed by $N_0$ for such $n$. 
Since $S\bar{w}\overset{(\ref{equation-S=T})}{=}T\bar{w}=\lambda \bar{w}$  in $\pi_2$  and $\lambda\neq0$, $\frac{1}{\lambda^n}S^nw$ is still a lifting of $\bar{w}$, so we may assume the chosen lifting $w$ is fixed by $N_0$. Since $\bar{w}$ is fixed by $I_1$, we have (in particular) $(h-1)\bar{w}=0$ for  any $h\in T_1$, hence $(h-1)w\in \pi_1$.  By Lemma \ref{lemma-S2}(v), there exists $n_h\gg0$ such that $S^{n_h}(h-1)w=0$. The representation $V$ being smooth, $\langle T_1.w\rangle$ is finite dimensional, so we may find $n$ large enough so that $S^n(h-1)w=0$ for all $h\in T_1$. But, Lemma \ref{lemma-S2}(iii) implies 
\[0=S^n(h-1)w=(h-1)S^nw,\]
so that $S^nw$ is fixed by $I_1\cap P$. Again, by Lemma \ref{lemma-S2}(iv), up to enlarge $n$, $S^nw$ is fixed by $I_1$. Replacing $w$ by $\frac{1}{\lambda^n}S^nw$, we obtain a lifting $w$ of $\bar{w}$ which is fixed by $I_1$.

Next, consider the vector $(S-\lambda)w$ which belongs to $\pi_1$ as its image in $\pi_2$ is zero. By Lemma \ref{lemma-S2}(v) again, there exists $n\gg0$ such that $S^n(S-\lambda)w=0$. Replacing $w$ by $\frac{1}{\lambda^n}S^nw$, we get a lifting $w$ of $\bar{w}$  satisfying $Sw=\lambda w$.
 
Summarizing,  we obtain a lifting $w\in V^{I_1}$ of $\bar{w}$ which is an eigenvector of $\cH$ and satisfies $Sw=\lambda w$.    By  the proof of \cite[Lem. 4.1]{Pa07}
\[\tau:=\langle K\cdot w\rangle \subset V\]
is irreducible and isomorphic to $\sigma$. Moreover, the fact $Sw=\lambda w$ implies that the $G$-morphism $\cInd_{KZ}^{G}\tau\otimes\chi\circ\det\ra V$ (here $\chi$ is the character appeared at the beginning of the proof) induced by Frobenius reciprocity must factor through  (note that $Sw=Tw$ by (\ref{equation-S=T}))
\[\phi:\cInd_{KZ}^{G}\tau/(T-\lambda)\otimes \chi\circ\det\ra V.\]  
This shows that $V$ contains   $\pi_1$ as a sub-representation and hence splits.
\end{proof}
 
\textbf{Acknowledgements} 
The author would like to thank V. Pa\v{s}k\=unas for some useful discussion and C. Breuil for several comments on the first version of the paper.


  \[\underline{\hspace{6cm}}\]

\bigskip

\begin{thebibliography}{99}
\bibitem{BL}
 {L.~Barthel and R.~Livn\'e},  `Irreducible modular representations of $\GL_2$ of a local field',  {\em Duke Math. J. }75 (1994) 261-292.


\bibitem{Be}
{L. Berger}, `Central characters for smooth irreducible modular representations of $\GL_2(\Q_p)$', {\em Rend. Sem. Mat. Univ. Padova} 128 (2012) 1-6.

 

  \bibitem{Br03}
 {C.~Breuil}, `Sur quelques repr\'esentations modulaires et $p$-adiques de $\GL_2(\Q_p)$: I',  {\em Compositio Math.} 138 (2003) 165-188.

 
 
\bibitem{BD11}
 {\bibname C.~Breuil and F.~Diamond}, {`Formes modulaires de Hilbert modulo $p$ et valeurs d'extensions galoisiennes'},  {\em Ann. Scient. de l'E.N.S.} 47 (2014) 905-974.

 

\bibitem{BP} {C. Breuil and V. Pa\v{s}k\={u}nas}, {`Towards a mod $p$ Langlands correspondence for $\GL_2$'},  {\em Memoirs of Amer. Math. Soc.} 216 (2012).

 
\bibitem{BH}  {W. Bruns and J. Herzog}, {\em Cohen-Macaulay rings (revised edition)}, Cambridge University Press, 1998.


\bibitem{BDJ}
 {K.~Buzzard, F.~Diamond  $\&$ F. ~Jarvis},  {`On Serre's conjecture for mod $\ell$ Galois representations over totally real fields'}, {\em Duke Math. J.} 55 (2010) 105-161.
 

\bibitem{Em1}  {M. Emerton}, {`Ordinary parts of admissible representations of $p$-adic reductive group I. Definition and first properties'},  {\em Ast\'erisque} 331 (2010) 355-402. 

\bibitem{Em2}  {M. Emerton}, {`Ordinary parts of admissible representations of $p$-adic reductive group II. Derived functors'}, {\em Ast\'erisque} 331 (2010) 403-459. 

\bibitem{Em3}  {M. Emerton}, {`Local-global compatibility in the $p$-adic Langlands programme for $\GL_{2/\Q}$'}, preprint (2011)

\bibitem{EGS}  { M. Emerton, T. Gee and D. Savitt}, `Lattices in the cohomology of Shimura curves',  {\em Invent. Math.} 200 (2015) 1-96.

 
\bibitem{EP}
 {M. Emerton and V. Pa\v{s}k\={u}nas},  `On effaceability of certain $\delta$-functors',  {\em Ast\'erisque} 331 (2010) 439-447.

\bibitem{FL}
{J.-M.~Fontaine $\&$ G. ~Laffaille},  {`Construction de repr\'esentations $p$-adiques'},  {\em Ann. Scient. E.N.S.} 15 (1982) 547-608.

\bibitem{Gee11}
{T.~Gee},  {`On the weights of mod $p$ Hilbert modular forms'},  {\em Invent. Math.} 184 (2011) 1-46.


 \bibitem{Hu12}  {Y. Hu}, `Diagrammes canoniques et repr\'esentations modulo $p$ de $\GL_2(F)$', {\em J. Inst. Math. Jussieu} 11 (2012)  67-118.

\bibitem{Hu13}  {Y. Hu}, `Valeurs sp\'eciales de param\`etres de diagrammes de Diamond',  {\em Bull. Soc. Math. France}, to appear.

 
\bibitem{Ko} {J. Kohlhaase}, `Smooth duality in natural characteristic', preprint 2014.

 

 

\bibitem{Pa07} {V. Pa\v{s}k\={u}nas}, `On the restriction of representations of $\GL_2(F)$ to a Borel subgroup', {\em Compositio Math.} 143 (2007) 1533-1544.

\bibitem{Pa10}
 {V.~Pa\v{s}k\={u}nas}, `Extensions for supersingular representations of $\GL_2(\Q_p)$',   {\em Ast\'erisque} 331 (2010) 317-353.

 
 
 
\bibitem{Sc} {B.~Schraen}, `Sur la pr\'esentation des repr\'esentations supersinguli\`eres de $\GL_2(F)$',  {\em J. reine angew. Math.} 704 (2015) 187-208. 

\bibitem{Se}  {J.-P.~Serre}, Cohomologie galoisienne, LNM 5, 2nd edition, 1997.

\bibitem{Vi} {M.-F. Vign\'eras}, `Repr\'esentations $p$-adiques de torsion admissibles',  {\em Number Theory, Analysis and Geometry: In memory of Serge Lang}, Springer (2011) 639-646.

 

\end{thebibliography}
 \end{document}